\documentclass[a4paper,11pt,oneside,reqno]{amsart}
\usepackage{amssymb,amsfonts,amsmath,amsthm,amscd}
\usepackage{graphicx, color}
\newcommand{\beq}{\begin{equation}}
\newcommand{\eeq}{\end{equation}}
\newcommand{\bea}{\begin{aligned}}
\newcommand{\eea}{\end{aligned}}
\newcommand{\bdm}{\begin{displaymath}}
\newcommand{\edm}{\end{displaymath}}
\newcommand{\barr}{\begin{array}}
\newcommand{\earr}{\end{array}}
\newcommand{\ben}{\begin{enumerate}}
\newcommand{\een}{\end{enumerate}}
\newcommand{\bde}{\begin{description}}
\newcommand{\ede}{\end{description}}

\numberwithin{equation}{section}
\newtheorem{teor}{Theorem}[section]
\newtheorem{prop}[teor]{Proposition}
\newtheorem{lem}[teor]{Lemma}

\newtheorem{Def}[teor]{Definition}

\newtheorem{rem}[teor]{Remark}

\newcommand{\Z}{\mathbb{Z}}
\newcommand{\R}{\mathbb{R}}

\newcommand{\N}{\mathbb{N}}

\newcommand{\PP}{\mathbb{P}}
\newcommand{\E}{\mathbb{E}}

\newcommand{\C}{\mathbb{C}}
\newcommand{\D}{\mathbb{D}}

\newcommand{\defi}{\stackrel{\text{def}}{=}}

\newcommand{\al}{\alpha}

\newcommand{\de}{\delta}

\newcommand{\X}{\mathcal{X}}
\newcommand{\M}{\mathcal{M}}

\newcommand{\lb}{\left (}
\newcommand{\rb}{\right )}
\newcommand{\lbb}{\left [}
\newcommand{\rbb}{\right ]}
\newcommand{\labs}{\left |}
\newcommand{\rabs}{\right |}
\newcommand{\lbrb}[1]{\lb #1 \rb}

\newcommand{\labsrabs}[1]{\labs#1\rabs}

\newcommand{\langlerangle}[1]{\langle#1\rangle}

\newcommand{\Pbb}[1]{\Pb\lb #1\rb}
\newcommand{\Ebb}[1]{\Eb\lbb #1\rbb}

\newcommand{\Eb}{\mathbb{E}}
\newcommand{\Nb}{\mathbb{N}}

\newcommand{\Pb}{\mathbb{P}}

\newcommand{\Bc}{\mathcal{B}}

\newcommand{\Mcc}{\mathcal{M}}

\newcommand{\Xc}{\mathcal{X}}

\begin{document}

\title[Fluctuation limits of a locally regulated population]{Fluctuation Limits of a locally regulated population and generalized Langevin equations} 

\author{Mladen Savov and Shi-dong Wang}
\address{M. Savov\\Institute of Mathematics and Informatics\\Bulgarian Academy of Sciences\\15 Noemvri Str.\\ 1040 - Sofia, Bulgaria}
\email{m.savov@reading.ac.uk}

\address{ S.-D. Wang\\Corresponding author: Department of Statistics\\University of Oxford
\\1 South Parks Road
\\ Oxford, OX1 3TG, UK}
\email{shidong.wang@stats.ox.ac.uk}

\keywords{Interacting measure-valued processes; regulated population,  fluctuation limit; Langevin equation; stationary distribution.} 

\subjclass{60K35; 60F05; 92D25} 

\thanks{S.-D. Wang is supported by a EPSRC Grant EP/I01361X/1 at the University of Oxford and was supported by a 
Hausdorff Scholarship while at the University of Bonn.}

\begin{abstract}
We consider a locally regulated spatial population model introduced by Bolker and Pacala. Based on the deterministic approximation studied by Fournier and M\'el\'eard, we prove that the fluctuation theorem holds under some mild moment conditions. The limiting process is shown to be an infinite-dimensional Gaussian process solving a generalized Langevin equation. In particular, we further consider its properties in one-dimension case, which is characterized as a time-inhomogeneous Ornstein-Uhlenbeck process. 
\end{abstract}

\maketitle

\section{Introduction} \label{section one}

It is well known that branching processes have been widely used to model the evolution in biological populations. If, in addition, the individuals are assumed to follow some independent motions (like Brownian motion or random walks), the system can be approximated by the so-called Dawson-Watanabe superprocess (refer to \cite{Dawson_superprocess,Etheridge_introductory,Li_superprocess}). The most common feature of these processes is that branching and spatial motion are independent.

Since individuals can reproduce, mutate and die in varying rates according to their different spatial characteristics (phenotypes), one reasonable improvement we can make is to add spatial components to both branching and dispersal parameters. Nevertheless, the spatial-dependent components destroy the independence between branching and dispersal while bringing us abundant information from the phenotypic point of view, and even though, the model is still deficient: such as in the finite-dimensional branching process model, the populations either die out or escape to infinity, depending on the mean matrix of the offspring distribution. The model thus can not predict a non-trivial equilibrium which actually happens quite often in the biological world.
Bolker and Pacala \cite{Bolker_Pacala} propose a self-regulated model which attains the above two improved features. By employing the idea of the ordinary logistic growth equation, they introduce a competition term in the density-dependent populations, which can help the system to attain equilibria under specific conditions. However, the loss of branching property can also cause some new technical difficulties when we study some properties such as weak convergence from branching particle systems to a continuum limit.

Law and Dieckmann \cite{law_Dieckmann} study this model in parallel with Bolker and Pacala \cite{Bolker_Pacala}. We simply call it BPDL model. In recent years, this model has been extensively studied in papers such as  Etheridge \cite{Etheridge},  Fournier and M\'el\'eard \cite{Founier _Meleard}, Champagnat \cite{Champagnat}, Lambert \cite{Lambert}, Dawson and Greven \cite{DawsonGreven}. Etheridge \cite{Etheridge} studies two diffusion limits, one is a \emph{stepping stone version} of the BPDL model (interacting diffusions indexed by $\Z^d$) and another is a \emph{superprocess version} of it. In that paper, sufficient conditions are given for survival and local extinction. Fournier and M\'el\'eard \cite{Founier _Meleard} formulate a pathwise construction of the BPDL process in terms of Poisson point processes. Under the finiteness of third moment condition, they rigorously obtain a deterministic approximation (law of large numbers) of the BPDL processes. Our work is based on the formalization of Fournier and 
M\'el\'eard \cite{Founier _Meleard}. In the papers Champagnat \cite{Champagnat}, Champagnat and M\'el\'eard \cite{ChamMeleard}, Dawson and Greven \cite{DawsonGreven}, they investigate long term behaviour of respective populations by the method of multiple time scales analysis.

In this paper we aim to present and prove the fluctuation theorem in a general framework set by Fournier and M\'el\'eard \cite{Founier _Meleard}, which could be applied in the derivative models studied by the referred authors. As for a sequence of density-dependent population processes with only finite-many types, Kurtz \cite{Kurtz} proves its central limit theorem, which is characterized by some finite-dimensional diffusion process. As for infinite-dimensional population models, Gorostiza and Li \cite{Gorostiza_Li} prove the high-density fluctuations of a branching particle system with  immigration, where they use the classical Laplace transform method owing to the branching property. In our case, this approach doesn't work anymore due to the loss of branching property.

The remainder of the paper is structured as follows.
In Section \ref{section two}, we briefly describe the model and give some preliminary results. More precisely, we recall the law of large numbers of the BPDL processes proved by Fournier and M\'el\'eard \cite{Founier _Meleard}.
In Section \ref{section three}, we build the fluctuation theorem and prove the tightness and finite-dimensional convergence based on some moment estimates in subsequent sections.
In Section \ref{section four}, in order to better understand the limiting process, we show it to be the solution of an infinite-dimensional inhomogeneous Langevin equation, which can be viewed as evolving in a deterministic medium. 
In Section \ref{section five}, we consider a degenerate case, the one dimensional version of the fluctuation limit. A precise characterization of the fluctuation diffusion is given as a time-inhomogeneous Ornstein-Uhlenbeck process. We study its stationary distribution as well.


\section{Model}\label{section two}
\subsection{Notation and description of the process}
Following \cite{Bolker_Pacala}, we assume that the population at time $t$ is composed of a finite number $I(t)$ of individuals characterized by their phenotypic traits $x_1(t), \cdots, x_{I(t)}(t)$ taking values in a compact subset $\X$ of $\R^d$.\\
We denote by $\M_F(\X)$ the set of finite measures on $\X$ (including negative-valued measures). Let $\M_a(\X)\subset \M_F(\X)$ be the set of counting measures on $\X$:
\[
\M_a(\X)=\left\{\sum\limits_{i=1}^n\de_{x_i}: x_1,\cdots, x_n\in\X, n\in\N\right\}.
\]
Then, the population process can be represented as:
\[
\nu_t=\sum\limits_{i=1}^{I(t)}\delta_{X_i(t)}.
\]
Let $\Bc(\X)$ denote the totality of functions on $\X$ that are  bounded measurable. Let $C^{\infty}(\X)$ denote the space of infinitely differentiable functions on $\R^d$ with support contained in $\X$. Let $\mathcal{S}(\R^d)$ denote the Schwartz space of (infinitely differentiable, rapidly decreasing) testing functions on $\R^d$ whose topological dual space is $\mathcal{S}'(\R^d)$, and $\langle \cdot, \cdot\rangle$ the canonical bilinear form on $\mathcal{S}'(\R^d)\times\mathcal{S}(\R^d)$. When $\mu\in\mathcal{S}'(\R^d)$  is a (signed) measure, then $\langle\mu,\phi\rangle=\int\phi d\mu,\,\phi\in\mathcal{S}(\R^d)$. With a slight abuse of notation we denote by $\mathcal{S}'(\X)\subset \mathcal{S}'(\R^d)$ the subset of tempered distributions $\psi\in\mathcal{S}'(\R^d)$ which satisfy $\langlerangle{\phi,\psi}=0$, for any $\phi\in\mathcal{S}(\X^c)$, i.e. $Supp~ \phi\cap \X=\emptyset$. Note that $\Mcc_F\lbrb{\Xc}\subset \mathcal{S}'\lbrb{\Xc}$ which follows immediately from the definition of $\mathcal{S}\lbrb{\Xc}$.\\
Let's specify the population processes $(\nu_t^n)_{t>0}$ by introducing a sequence of biological parameters, for n$\in\N$:
\begin{itemize}
  \item $b_n(x)$ is the rate of birth from an individual with trait $x$.
  \item $d_n(x)$ is the rate of death of an individual with trait $x$ because of ``aging''.
  \item $\al_n(x,y)$ is the competition kernel felt by some individual with trait $x$ from another individual with trait $y$.
  \item $D_n(x,dz)$ is the children's dispersion law from the mother with trait $x$. Assume that
  \beq\label{dispersal kernel}
  D_n(x,dz)=m_n(x,z)dz.
  \eeq  
\end{itemize}
Here, $m_n(x,z)$ is the probability density for mutation variation, which satisfies
  \[
  \int_{z\in\R^d,x+z\in\X}m_n(x,z)dz=1.
  \]
Fournier and M\'el\'eard \cite{Founier _Meleard} have formulated a pathwise construction of the BPDL process $\{(\nu_t^n)_{t\geq0}; n\in\N\}$ in terms of Poisson random measures and justified its infinitesimal generator defined for any $\Phi\in \Bc(\M_a(\X))$:
\beq\label{generator BPDL}
\bea
L_0^n\Phi(\nu)=&\int_{\X}\nu(dx)\int_{\R^d}\Big(\Phi(\nu+\de_{x+z})-\Phi(\nu)\Big)b_n(x)D_n(x,dz)\\
             &+\int_{\X}\nu(dx)\Big(\Phi(\nu-\de_x)-\Phi(\mu)\Big)\left(d_n(x)+\int_{\X}\al_n(x,y)\nu(dy)\right).
\eea
\eeq
The first term is used to model birth events, while the second term which is nonlinear is interpreted as natural death and competing death.\\
Instead of studying the original BPDL processes defined by \eqref{generator BPDL}, our goal is to study the rescaled processes
 \beq\label{BPDL process scaled} X_t^n:=\frac{\nu_t^n}{n}, \qquad t\geq0
 \eeq
since it provides us a macroscopic approximation when we take the large population limits (we will see later, the initial population is proportional to $n$ in some sense).
The infinitesimal generator of the rescaled BPDL process has the form, for any $\Phi\in \Bc(\M_F(\X))$:
\beq\label{generator BPDL rescaled}
\bea
L^n\Phi(\mu)=&\int_{\X}n\mu(dx)\int_{\R^d}\Big(\Phi(\mu+\frac{\de_{x+z}}{n})-\Phi(\mu)\Big)b_n(x)D_n(x,dz)\\
             &+\int_{\X}n\mu(dx)\Big(\Phi(\mu-\frac{\de_x}{n})-\Phi(\mu)\Big)\left(d_n(x)+\int_{\X}\al_n(x,y)n\mu(dy)\right).
\eea
\eeq

\subsection{Preliminary results}
Let's denote by (A) the following assumptions:

 (A1) There exist $b(x),\, d(x),\,f_i(x), \,g_i(x)\in C^{\infty}(\Xc)\, (1\leq i\leq m^d)$ and $m(x,z)\in C^{\infty}(\X\times (\X-\X))$ such that, for $x,\,y\in\X, z\in\X-\X, \,n\in\N$,
 \[\bea & 0<b_n(x)\equiv b(x),\qquad 0<d_n(x)\equiv d(x), \qquad m_n(x,z)\equiv m(x,z),\\
        &0<\al_n(x,y)=\frac{\al(x,y)}{n}, \qquad\alpha(x,y)=\sum_{i=1}^{m^d}f_i(x)g_i(y).
   \eea
 \]

 (A2) $ b(x)-d(x)>0$.\\

 The first assumption implies that there exist constants $\bar{b}, \,\bar{d}, \,\bar{\al}$ such that $b(x)\leq\bar{b}, \,d(x)\leq\bar{d}, \,\al(x,y)\leq\bar{\al}$. The assumption on $\alpha(x,y)$ being the sum above is purely technical for the proof of Lemma \ref{lem:estimateInteraction} and the choice of $m^d$ irrelevant as it can be any positive integer. It seems that the technical restriction on $\alpha$ is very hard to remove at the level of CLT when there is an interaction. However, one easily sees that any smooth function $\alpha$ can be approximated in the supremum norm by expressions of $\alpha$ as in (A1). This then is perfectly suitable for any practical or numerical purpose.

 By neglecting the high order moment, Bolker and Pacala \cite{Bolker_Pacala} use the ``moment closure'' procedure to approximate the stochastic population processes. As we can see from the generator form \eqref{generator BPDL rescaled}, it should be enough to ``close'' the second order moment due to the quadratic nonlinear term. Actually the result proved by Fournier and M\'el\'eard still holds under a second moment condition $\sup\limits_{n\geq1}\E\langle X_0^n,1\rangle^2<\infty$. We only recall their result here without mentioning the detailed proof repeatedly.

 \begin{teor}[Convergence to a nonlinear integro-differential equation]\label{Theorem LLN}
Under the assumption (A1) consider the sequence of processes $(X_t^n)_{t\geq 0}$ defined in \eqref{BPDL process scaled}. Suppose that $(X_0^n)$ converges in law to some deterministic finite measure $X_0\in\M_F(\X)$ as $n\to\infty$ and satisfies $\sup\limits_{n\geq1}\E\langle X_0^n,1\rangle^2<\infty$.

Then the sequence of processes $(X_t^n)_{t\geq 0}$ converges in law as $n\to\infty$, on $\D([0,\infty),\M_F(\X))$, to a deterministic measure-valued process $(X_t)_{t\geq0}\in\C([0,\infty),\M_F(\X))$, where $(X_t)_{t\geq0}$ is the unique solution satisfying
\beq\label{LLN limit}
\bea
\langle X_t,\phi\rangle =&\langle X_0,\phi\rangle+\int_0^tds\int_{\X}X_s(dx)b(x)\int_{\R^d}\phi(x+z)D(x,dz)\\
\qquad &-\int_0^tds\int_{\X}X_s(dx)\phi(x)\Big(d(x)+\int_{\X}\al(x,y)X_s(dy)\Big).
\eea
\eeq
for any $\phi\in\Bc(\Xc)$ and
\beq\label{LLN Limit_initial}
\sup\limits_{t\in[0,T]}\langle X_t,1\rangle<\infty.
\eeq

 \end{teor}

Finally, it comes to a natural question: how does $(X_t^n)_{t\geq 0}$ fluctuate around the macroscopic limit $(X_t)_{t\geq0}$ given above? A natural candidate to be investigated could be the centralized sequence of processes:
\beq\label{BPDL process rescaled}
Y_t^n:=\frac{\nu_t^n-nX_t}{\sqrt{n}}=\sqrt{n}(X_t^n-X_t).
\eeq
In the following proposition, we will give some martingale properties of the processes $(Y_t^n)_{t\geq0}$, which will play a key role in the proof of the main theorem.

\begin{prop}\label{Martingale BPDL Rescaled}
Admit the same assumptions as in Theorem \ref{Theorem LLN}. For fixed $n\in\N$ and $\phi\in\Bc(\Xc)$, the process
\beq\label{martingale BPDL rescaled_Y^n}
\bea
M_t^n(\phi):=&\langle Y_t^n,\phi\rangle-\langle Y_0^n,\phi\rangle-\int_0^tds\int_{\X}Y_s^n(dx)b(x)\int_{\R^d}\phi(x+z)D(x,dz)\\
&+\int_0^tds\int_{\X}\phi(x)d(x)Y_s^n(dx)\\
&+\sqrt{n}\int_0^tds\int_{\X}X_s^n(dx)\phi(x)\int_{\X}\al(x,y)X_s^n(dy)\\
&-\sqrt{n}\int_0^tds\int_{\X}X_s(dx)\phi(x)\int_{\X}\al(x,y)X_s(dy)
\eea
\eeq
is a c\`adl\`ag square integrable martingale with quadratic variation
\beq\label{martingale BPDL rescaled_variation}
\bea
\langle M_{\cdot}^n(\phi)\rangle_t&=\int_0^tds\int_{\X}X_s^n(dx)b(x)\int_{\R^d}\phi^2(x+z)D(x,dz)\\
&\qquad+\int_0^tds\int_{\X}X_s^n(dx)\phi^2(x)\Big(d(x)+\int_{\X}\al(x,y)X_s^n(dy)\Big).
\eea
\eeq
\end{prop}

\begin{proof} Recall the generator \eqref{generator BPDL rescaled}, for bounded measurable functional $\Phi$ on $\M_F(\X)$, the process
\[
\Phi(X_t^n)-\Phi(X_0^n)-\int_0^t L^n \Phi(X_s^n)ds
\]
is a c\`adl\`ag square integrable martingale. If we take $\Phi(\mu)=\langle \mu, \phi\rangle,\,\forall\phi\in \Bc(\X)$, one obtains that
\beq\label{martingale for X}
\bea
N_t^n(\phi):=&\langle X_t^n,\phi\rangle-\langle X_0^n,\phi\rangle-\int_0^tds\int_{\X}X_s^n(dx)b(x)\int_{\R^d}\phi(x+z)D(x,dz)\\
&+\int_0^tds\int_{\X}X_s^n(dx)\phi(x)\Big(d(x)+\int_{\X}\al(x,y)X_s^n(dy)\Big)
\eea
\eeq
is a c$\grave{a}$dl$\grave{a}$g martingale.
By applying It$\hat{o}$'s formula to $\langle X_t^n,\phi\rangle^2$, we have
\[
\bea\langle X_t^n,\phi\rangle^2&-\langle X_0^n,\phi\rangle^2-2\int_0^tds\langle X_s^n,\phi\rangle\int_{\X}X_s^n(dx)\Big\{b(x)\int_{\R^d}\phi(x+z)D(x,dz)\\
&-\phi(x)\Big(d(x)+\int_{\X}\al(x,y)X_s^n(dy)\Big)\Big\} -\langle N_\cdot^n(\phi)\rangle_t
\eea
\]
is a martingale.
On the other hand, if we take $\Phi(\mu)=\langle \mu, \phi\rangle^2$,  it follows that
\beq\label{martingale for X^2}
\bea
\langle X_t^n,\phi\rangle^2&-\langle X_0^n,\phi\rangle^2-\int_0^t L^n\Phi(X_s^n)ds\\
=&\langle X_t^n,\phi\rangle^2-\langle X_0^n,\phi\rangle^2-2\int_0^tds\langle X_s^n,\phi\rangle\int_{\X}X_s^n(dx)\Big\{b(x)\int_{\R^d}\phi(x+z)D(x,dz)\\
&-\phi(x)\Big(d(x)+\int_{\X}\al(x,y)X_s^n(dy)\Big)\Big\} \\
&-\frac{1}{n}\int_0^tds\int_{\X}X_s^n(dx)b(x)\int_{\R^d}\phi^2(x+z)D(x,dz)\\
&-\frac{1}{n}\int_0^tds\int_{\X}X_s^n(dx)\phi^2(x)\Big(d(x)+\int_{\X}\al(x,y)X_s^n(dy)\Big)
\eea
\eeq
is a martingale.
By comparing the two decompositions of the semimartingale $\langle X_t^n,\phi\rangle^2$ above, one obtains that
\beq\label{quadratic for X}
\bea
\langle N_\cdot^n(\phi)\rangle_t
&=\frac{1}{n}\int_0^tds\int_{\X}X_s^n(dx)b(x)\int_{\R^d}\phi^2(x+z)D(x,dz)\\
&+\frac{1}{n}\int_0^tds\int_{\X}X_s^n(dx)\phi^2(x)\Big(d(x)+\int_{\X}\al(x,y)X_s^n(dy)\Big). \eea
\eeq
Owing to \eqref{martingale for X} and \eqref{LLN limit}, do the operation $\big(\langle X_t^n,\phi\rangle-\langle X_t,\phi\rangle\big)$ and let $M_t^n(\phi):=\sqrt{n}N_t^n(\phi)$, to conclude the proof by the definition of $(Y_t^n)$ in $\eqref{BPDL process rescaled}$.
\end{proof}

\section{Fluctuation theorem}\label{section three}
In this section, our aim is to study the asymptotic behavior of $(Y_t^n)_{t\geq0}$ as $n\to\infty$. The following theorem, the main result of the paper, shows that $(Y_t^n)_{t\geq0}$ indeed converges to the unique solution of a martingale problem.

In the following sections, we will use the notation given in Section \ref{section two} without declaration. We always assume that assumption (A1) holds.

\begin{teor}\label{Theorem CLT}
Admit assumption (A1) and suppose that there exists a deterministic finite nonnegative measure $X_0\in\M_F(\X)$ such that $Y_0^n=\sqrt{n}(X_0^n-X_0)$ satisfies for some $\delta>0$
\beq\label{moment condition}
\sup\limits_{n\geq1}\lbrb{\sup\limits_{\phi\in\Bc(\Xc); {\|\phi\|}_{\infty}\leq 1}\E\labsrabs{\langle Y^n_0,\phi\rangle}^{4+\delta}}<\infty.
\eeq
Suppose that $(Y_0^n)$  converges in law to a finite (maybe random) measure $\gamma$ as $n\to\infty$.
Then, the process $(Y_t^n)_{t\geq0}$ converges in law as $n\to\infty$ on $\D([0,\infty),\mathcal{S}'(\X))$ to a process $(Y_t)_{t\geq0}\in\C([0,+\infty),\mathcal{S}'(\X))$
 where $(Y_t)_{t\geq0}$ is the unique solution satisfying
 \beq\label{CLT limit}
 \bea
 \langle Y_t,\phi\rangle=\langle \gamma,\phi\rangle&+\int_0^tds\,\Big\langle Y_s,\, b(\cdot)\int_{\R^d}\phi(\cdot+z)D(\cdot,dz)\Big\rangle\\
 &-\int_0^t ds\, \Big\langle Y_s,\, d(\cdot)\phi(\cdot)\Big\rangle\\
 &-\int_0^t ds\,\Big\langle Y_s,\, \int_{\X}\al(x,\cdot)\phi(x)X_s(dx)\Big\rangle\\
 &-\int_0^t ds\, \Big\langle Y_s, \, \phi(\cdot)\int_{\X}\al(\cdot,y)X_s(dy)\Big\rangle\\
 &+M_t(\phi)
 \eea
 \eeq
 for any $\phi\in\mathcal{S}(\R^d)$.
Here, $(X_t)_{t\geq0}$ is the solution defined by the deterministic nonlinear equation $\eqref{LLN limit}$, while $M_t(\phi)$ is a continuous martingale with quadratic variation
\beq\label{CLT variation}
\bea
\langle M_{\cdot}(\phi)\rangle_t=&\int_o^tds\int_{\X}X_s(dx)b(x)\int_{\R^d}\phi^2(x+z)D(x,dz)\\
                        &+\int_0^tds\int_{\X}X_s(dx)\phi^2(x)\Big(d(x)+\int_{\X}\al(x,y)X_s(dy)\Big).
\eea
\eeq
\end{teor}

\begin{rem}
  The argument above makes essentially use of the initial moment \eqref{moment condition} and the initial convergence condition. This condition fulfils the assumptions needed in Theorem \ref{Theorem LLN}. Therefore, the law of large number limit $(X_t)_{t\geq0}$ is well defined (see Lemma \ref{Moment estimate X_lemma}).
\end{rem}
\begin{rem}
To avoid confusion, let us point out that we first prove that $Y$ is an $\mathcal{S}'(\R^d)$-valued process, and $\Mcc_F\lbrb{\Xc}\subset\mathcal{S}'(\R^d)$. Subsequently, we show that $Y$ is an $\mathcal{S}'(\Xc)$-valued process. Therefore, there is no special argument related to the structural properties of $\mathcal{S}'(\Xc)$. Note that M\'el\'eard \cite{Meleard1997} studies the convergence of fluctuations associated with Boltzmann equations in a weighted Sobolev space with a ``Sobolev embedding'' technique. 
\end{rem}

Proving the theorem is the content of the following sections.


\subsection{Moment estimates and tightness}
The tightness criterion is established for semimartingales based on the moment estimates (see \cite{Etheridge_introductory}). Our first two lemmas give the uniform second order moment estimates for a sequence of processes over finite time intervals.

\begin{lem}\label{Moment estimate X_lemma}
 Suppose that a sequence of random variables $(Y_0^n)$ in $\M_F(\X)$ satisfies the same condition as in Theorem \ref{Theorem CLT}.
Then,  $X_0^n\stackrel{\text{in law}}{\longrightarrow} X_0$, as $n\to\infty$, and
\beq\label{moment condition Y_O}
\sup\limits_{n\geq1}\lbrb{\sup\limits_{\phi\in\Bc(\Xc); {\|\phi\|}_{\infty}\leq 1}\E\langle Y^n_0,\phi\rangle^2}<\infty
\eeq
 and
\beq\label{moment estimate X_0}
\sup\limits_{n\geq1}\E\langle X_0^n, 1\rangle^2<\infty.
\eeq
Hence, Theorem \ref{Theorem LLN} holds.

In particular, for any $T<\infty$, there exists a constant $C_T^{(1)}>0$ such that
\beq\label{moment estimate X}
\sup\limits_{n\geq1}\E[\sup\limits_{0\leq t\leq T}\langle X_t^n,1\rangle^2]\leq C_T^{(1)}.
\eeq
\end{lem}
\begin{proof} In fact, the convergence from $(X_0^n)$ to $X_0$ in law can be implied by the convergence from $(Y_0^n)$ to $\gamma$.
By H\"{o}lder inequality, we easily get that
\beq
\sup\limits_{n\geq1}\lbrb{\sup\limits_{\phi\in\Bc(\Xc); {\|\phi\|}_{\infty}\leq 1}\E\langle Y^n_0,\phi\rangle^2}\leq \sup\limits_{n\geq1}\lbrb{\sup\limits_{\phi\in\Bc(\Xc); {\|\phi\|}_{\infty}\leq 1}\E\labsrabs{\langle Y^n_0,\phi\rangle}^{4+\delta}}^{\frac{2}{4+\delta}}<\infty.
\eeq

On the other hand,
  because of the definition of $(Y_0^n)$ as in \eqref{BPDL process rescaled}, we obtain that
  \[
  \sup\limits_{n\geq1}\E\langle X_0^n, 1\rangle^2\leq 2\langle X_0, 1\rangle^2+ 2\sup\limits_{n\geq1}\frac{1}{n}\E\langle Y_0^n, 1\rangle^2<\infty.
  \]
  Now the proof of the moment estimate \eqref{moment estimate X} follows immediately from \cite[Theorem 3.1]{Founier _Meleard} by applying the Gronwall's lemma.
\end{proof}
\begin{lem}\label{Moment estimate Y_lemma}
 Suppose that a sequence of random variables $Y_0^n\in\M_F(\X)$ satisfies \eqref{moment condition}.
Then, for any $T<\infty$, there exists a constant  $C_T^{(2)}>0$ such that
\beq\label{moment estimate Y}
\sup\limits_{n\geq1}\lbrb{\sup\limits_{\phi\in\Bc(\Xc); {\|\phi\|}_{\infty}\leq 1}\E[\sup\limits_{0\leq t\leq T}\langle Y_t^n,\phi\rangle^2]}\leq C_T^{(2)}.
\eeq

\end{lem}
\begin{proof}

From Proposition \ref{Martingale BPDL Rescaled}, by H\"older inequality, one obtains that
\beq
\bea
\langle Y_t^n,\phi\rangle^2&\leq 2\Bigg\{\langle Y_0^n,\phi\rangle^2+t\int_0^t\Big(\int_{\X}Y_s^n(dx)b(x)\int_{\R^d}\phi(x+z)D(x,dz)\Big)^2ds\\
&\qquad+t\int_0^t\Big(\int_{\X}Y_s^n(dx)d(x)\phi(x)\Big)^2ds\\
&\qquad+nt\int_0^t\Big(\int_{\X}X_s^n(dx)\phi(x)\int_{\X}\al(x,y)X_s^n(dy)\\
&\qquad\qquad\qquad-\int_{\X}X_s(dx)\phi(x)\int_{\X}\al(x,y)X_s(dy)\Big)^2ds\\
&\qquad+\Big[M_t^n(\phi)\Big]^2\Bigg\}\\
&\leq 2\Bigg\{\langle Y_0^n,\phi\rangle^2+t\int_0^t\Big(\int_{\X}Y_s^n(dx)b(x)\int_{\R^d}\phi(x+z)D(x,dz)\Big)^2ds\\
&\qquad+t\int_0^t\Big(\int_{\X}Y_s^n(dx)d(x)\phi(x)\Big)^2ds\\
&\qquad+2t\int_0^t\Big(\int_{\X}Y_s^n(dx)\phi(x)\int_{\X}\al(x,y)X_s^n(dy)\Big)^2ds\\
&\qquad\qquad\qquad+2t\int_0^t\Big(\int_{\X}X_s(dx)\phi(x)\int_{\X}\al(x,y)Y_s^n(dy)\Big)^2ds\\
&\qquad+\Big[M_t^n(\phi)\Big]^2\Bigg\}.
\eea
\eeq
Since $\E[\sup\limits_{0\leq t\leq T}\langle Y_t^n,\phi\rangle^2]$ is a finite quantity thanks to $X$ and the definition of $Y^n$, for any fixed $T<\infty$, first take the supremum over time interval $[0,T]$, then take expectations on both sides. It follows that, for any $\phi\in \Bc(\X)$ satisfying $\|\phi\|_{\infty}\leq 1$,
\beq\label{moment estimate Y_Gronwall}
\bea
\E[\sup\limits_{0\leq t\leq T}\langle Y_t^n,\phi\rangle^2]
&\leq 2\E\langle Y_0^n,\phi\rangle^2+2T\bar{b}^2\int_0^T\E\Big(\int_{\X}Y_s^n(dx)\frac{b(x)}{\bar{b}}\int_{\X}\phi(x+z)D(x,dz)\Big)^2ds\\
&\qquad+2T\bar{d}^2\int_0^T\E\Big(\int_{\X}Y_s^n(dx)\frac{d(x)}{\bar{d}}\phi(x)\Big)^2ds\\
&\qquad+4T\bar{\al}^2\int_0^T\E\Big(\int_{\X}Y_s^n(dx)\phi(x)\int_{\X}\frac{\al(x,y)}{\bar{\al}}X_s^n(dy)\Big)^2ds\\
&\qquad+4T\bar{\al}^2\int_0^T\E\Big(\int_{\X}X_s(dx)\phi(x)\int_{\X}\frac{\al(x,y)}{\bar{\al}}Y_s^n(dy)\Big)^2ds\\
&\qquad+2\E\Big\{\sup\limits_{0\leq t\leq T}[M_t^n(\phi)]^2\Big\}\\
&\defi 2\E\langle Y_0^n,\phi\rangle^2+\text{{\bf I}}+\text{{\bf II}}+\text{{\bf III}}+\text{{\bf IV}+\text{{\bf V}}}.
\eea
\eeq
To the end, we give estimate of every term in the equation above separately. \\
As for term {\bf V}, by Doob's maximal inequality and \eqref{martingale BPDL rescaled_variation}, we have that
\beq
\bea
\text{{\bf V}}&\leq 2\times 4 \E\Big[M_{T}^n(\phi)^2\Big]\\
&\leq 8\E [\langle M_{\cdot}^n(1)\rangle_T]\\
&\leq 8(\bar{b}+\bar{d})\E\Big\{\int_0^T \sup\limits_{0\leq u\leq s}\langle X_u^n, 1\rangle ds\Big\}+ 8\bar{\al}\E\int_0^T \sup\limits_{0\leq u\leq s}\langle X_u^n,1\rangle^2ds\\
&\leq 8(\bar{b}+\bar{d}+\bar{\al}) T\cdot(C_T^{(0)}+C_T^{(1)}),
\eea
\eeq
where the last inequality is due to \eqref{moment estimate X}.\\
Since $\|\frac{b(x)}{\bar{b}}\int_{\R^d}\phi(x+z)D(x,dz)\|_{\infty}\leq 1,\,\|\frac{d(x)}{\bar{d}}\phi(x)\|_{\infty}\leq1$, one obtains that
\beq
\text{{\bf I}}+\text{{\bf II}}\leq 2T^2\big(\bar{b}^2+\bar{d}^2\big)\int_0^T\sup\limits_{\phi\in\Bc(\X); \|\phi\|_{\infty}\leq 1}\E[\sup\limits_{0\leq u\leq s}\langle Y_u^n,\phi\rangle^2]ds.
\eeq\\
Similarly, {\bf IV} can be bounded by $4T^2\bar{\al}^2C\int_0^T\sup\limits_{\phi\in\Bc(\X); \|\phi\|_{\infty}\leq 1}\E[\sup\limits_{0\leq u\leq s}\langle Y_u^n,\phi\rangle^2]ds$ with some constant $C$  since $\tilde\phi(y)=\int\limits_{\Xc}X_u(dx)\frac{\alpha(x,y)}{\bar{\alpha}}\phi(x)$ is supremum norm bounded by a constant due to the boundedness condition \eqref{LLN Limit_initial}.\\

Term {\bf III} is the source of all troubles. We estimate it via Lemma \ref{lem:estimateInteraction} to get 
\[{\bf III}\leq 4T\bar{\alpha}^2m^d\int_{0}^{T}\sup_{\phi\in\Bc(\X); \|\phi\|_{\infty}\leq 1}\Ebb{\langlerangle{\phi,Y^n_s}^2\langlerangle{1,X^n_s}^2} ds.\]
For any $A>0$ we have that 
\begin{align*}
{\bf III}\leq &4TA^2\bar{\alpha}^2m^d\int_{0}^{T}\sup_{\phi\in\Bc(\X); \|\phi\|_{\infty}\leq 1}\Ebb{\sup_{0\leq u\leq s}\langlerangle{\phi,Y^n_u}^2} ds\\
&+4T\bar{\alpha}^2m^d\int_{0}^{T}\sup_{\phi\in\Bc(\X); \|\phi\|_{\infty}\leq 1}\Ebb{\langlerangle{\phi,Y^n_s}^2\langlerangle{1,X^n_s}^2;\langlerangle{1,X^n_s}>A} ds
\end{align*}
From Lemma \ref{lem:finiteness} we get that with some $D>0$
\[\sup_{n\geq 1}4TA^2\bar{\alpha}^2m^d\int_{0}^{T}\Ebb{\langlerangle{\phi,Y^n_s}^2\langlerangle{1,X^n_s}^2;\langlerangle{1,X^n_s}>A} ds<D<\infty.\]
Therefore
\begin{align*}
{\bf III}\leq &4TA^2\bar{\alpha}^2m^d\int_{0}^{T}\sup_{||\phi||_{\infty}\leq 1}\Ebb{\sup_{0\leq u\leq s}\langlerangle{\phi,Y^n_u}^2} ds+D
\end{align*}

Let $G^n(T):=\sup\limits_{\phi\in\Bc(\X); \|\phi\|_{\infty}\leq 1}\E[\sup\limits_{0\leq t\leq T}\langle Y_t^n,\phi\rangle^2]$, by combining the estimates above and \eqref{moment estimate Y_Gronwall}, one obtains that
\beq
\bea
G^n(T)
\leq & 2\sup\limits_{n\geq 1}\sup\limits_{\phi\in\Bc(\X); \|\phi\|_{\infty}\leq 1}\E\langle Y_0^n,\phi\rangle^2+8(\bar{b}+\bar{d}+\bar{\al}) T\cdot(C_T^{(0)}+C_T^{(1)})\\
&+\Big(2T^2(\bar{b}^2+\bar{d}^2)+4T^2\bar{\al}^2C\Big)\int_0^T G^n(s)ds
\eea
\eeq
By Gronwall's lemma, we have that
\beq
\bea
G^n(T)\leq &\Big(2\sup\limits_{n\geq 1}\sup\limits_{\phi\in\Bc(\X); \|\phi\|_{\infty}\leq 1}\E\langle Y_0^n,\phi\rangle^2+8(\bar{b}+\bar{d}+\bar{\al}) T\cdot(C_T^{(0)}+C_T^{(1)})\Big)\\
&\cdot \exp \Big\{\Big(2T^2(\bar{b}^2+\bar{d}^2)+4T^2\bar{\al}^2C\Big)T\Big\}\\
&\defi C_T^{(2)}.
\eea
\eeq
Since $C_T^{(2)}$ is a $n$-independent constant,
the lemma follows by taking supremum over $n\in\N$ on both sides of the last inequality.
\end{proof}

Here we present some axillary results needed for the proofs above. 
\begin{lem}\label{lem:largedeviation}
With $X^n_s$ defined as above we have that for any $A>0$, $n\in\Nb_+$, $s>0$ and $j<[A]$
\begin{equation}\label{eq:fourthmoment}
\Pbb{\langlerangle{X^n_s,1}>A}\leq \sum_{m=0}^{nj}e^{\lbrb{n-\frac{m}{[A]}}\lbrb{s-\ln\lbrb{1+[A]\frac{1-1/[A]}{1+1/n}}}}\Pbb{\langlerangle{X^n_0,1}=\frac{m}{n}}+\Pbb{\langlerangle{X^n_0,1}>j}
\end{equation}
Next if $A=A(s)$ and $n$ are so big that $[A]\frac{1-1/[A]}{1+1/n}>\frac{A}{2}$, $s-\ln\lbrb{1+[A]/2}<0$ and take $j=[\beta A]$, for any $\beta\in(0,1)$, then
\begin{equation}\label{eq:fourthmoment1}
\Pbb{\langlerangle{X^n_s,1}>A}\leq e^{n\lbrb{1-\beta}\lbrb{s-\ln\lbrb{1+[A]/2}}}+\Pbb{\langlerangle{X^n_0,1}>[\beta A]}
\end{equation}
Finally, put $A=2e^{s+t}-2$, $t>0$, to get 
\begin{equation}\label{eq:fourthmoment2}
\Pbb{\langlerangle{X^n_s,1}>A}\leq e^{-n(1-\beta)t}+\Pbb{\langlerangle{X^n_0,1}>[\beta A]},
\end{equation}
where clearly $t\sim \ln A$, as $A\to\infty$.
\end{lem}
\begin{proof}
For each $n\in\Nb_+$ we use the obvious pathwise upper bound for the mass of the measure $X^{n}_s$ by the total mass of the rescaled pure birth process of birth rate $\bar b$ and initial random measure $\tilde\nu^n_0=\nu^n_0$, say $\tilde{X}^n_s=\tilde\nu^n_s/n$. Then clearly, for any $m>0$
\begin{align*}
& \Pbb{\langlerangle{X^n_s,1}>A\Big|\langlerangle{X^n_0,1}=\frac{m}{n}}\leq \Pbb{\langlerangle{\tilde{X}^n_s,1}>A\Big|\langlerangle{X^n_0,1}=\frac{m}{n}}\\
&=\Pbb{\langlerangle{\tilde{\nu}^n_s,1}>[A]n\Big| \langlerangle{\tilde{\nu}^0_n,1}=m},
\end{align*}
where $[A]$ stands for the smallest integer less than $A$. Note that, for any $m\in\Nb_+$
\[O_{m,n}(s)=\left\{\langlerangle{\tilde\nu^n_s,1}>[A]n\Big|\tilde\nu^n_0=m\right\}=\left\{\sum_{i=0}^{\min(0,[A]n-m)}T_{i}(m)\leq s \right\},\]
where $T_{i}(m)\sim Exp(m+i)$, $i\geq 1$, are independent random variables. Assume that $k=\min(0,[A]n-m)>0$. Then for any $\lambda>0$ using Markov's inequality we get
\[\Pbb{O_{m,n}(s)}\leq e^{\lambda s}e^{-\lambda \sum_{i=0}^{k}T_{i}(m)}=e^{\lambda s}\prod_{i=0}^{k}\frac{m+i}{m+i+\lambda}=e^{\lambda s}e^{\sum_{i=0}^{k}\ln{\lbrb{1-\frac{\lambda}{m+i+\lambda}}}}.\]
Using $\ln(1-x)\leq -x$ for $0\leq x\leq 1$ we get
\[\Pbb{O_{m,n}(s)}\leq e^{\lambda s}e^{-\lambda\sum_{i=0}^{k}\frac{1}{m+i+\lambda}}\leq e^{\lambda s}e^{-\lambda\int_{m+1+\lambda}^{m+1+k+\lambda}\frac{1}{x}dx}=e^{\lambda\lbrb{s-\ln\lbrb{\frac{m+k+1+\lambda}{m+1+\lambda}}}}.\]
Upon choosing $\lambda=Ck=C([A]n-m)$ for some $C>0$ we get
\begin{align*}
&\Pbb{O_{m,n}(s)}\leq e^{C\lbrb{[A]n-m} \lbrb{s-\ln\lbrb{\frac{(1+C)[A]n+1-m}{C[A]n+1}}}}\\
&=e^{C\lbrb{[A]n-m}\lbrb{s-\ln\lbrb{1+\frac{1-Cm/[A]n}{C+1/[A]n}}}}.
\end{align*}
Finally choose $C=1/[A]$ to get
\[\Pbb{O_{m,n}(s)}\leq e^{\lbrb{n-\frac{m}{[A]}}\lbrb{s-\ln\lbrb{1+[A]\frac{1-m/[A]^2n}{1+1/n}}}}.\]
Note that since by definition $[A]n-m\geq 0$ we see that $1-m/[A]^2n\geq 1-1/[A]$ and hence
\[\Pbb{O_{m,n}(s)}\leq e^{\lbrb{n-\frac{m}{[A]}}\lbrb{s-\ln\lbrb{1+[A]\frac{1-1/[A]}{1+1/n}}}} \]
Since upon conditioning on the total mass of the initial condition $\langlerangle{X^n_0,1}$ the probability of the set $\{\langlerangle{X^n_s,1}>A\}$ can be computed via the total probability formula. Therefore \eqref{eq:fourthmoment} follows and next \eqref{eq:fourthmoment1} and \eqref{eq:fourthmoment2} are deduced by mere substitution of the choice of $A$, $j$ in \eqref{eq:fourthmoment}.
\end{proof}
The next lemma allows us to handle suitable quantities.
\begin{lem}\label{lem:finiteness}
We have that for any fixed $ T>0$ there is $A>0$ such that
\begin{align*}
&\sup_{n\geq 1}\sup_{0\leq s\leq T}\sup_{\phi\in\Bc(\Xc);||\phi||_\infty \leq 1}\Ebb{\langlerangle{Y^n_s,\phi}^2\langlerangle{X^n_s,1}^2;\langlerangle{X^n_s,1}>A}<\infty
\end{align*}
\end{lem}
\begin{proof}
Clearly by the definition of $Y^n_s=\sqrt{n}\lbrb{X^n_s-X_s}$ and the fact that $\langlerangle{X_s,1}$ is a bounded deterministic process on any finite interval we get using H$\ddot{o}$lder's inequality and $(a+b)^2\leq 2a^2+2b^2$ that
\begin{align*}
&\Ebb{\langlerangle{Y^n_s,\phi}^2\langlerangle{X^n_s,1}^2;\labsrabs{\langlerangle{X^n_s,1}}>A}\\
\leq &2n\lbrb{1+\langlerangle{X_s,1}^2}\Ebb{\langlerangle{X^n_s,1}^4;\labsrabs{\langlerangle{X^n_s,1}}>A}
\end{align*}
By a standard formula for positive random variables we have that
\[\Ebb{X^4;X>A}\leq 4\int_{A}^{\infty}\Pbb{X>y}y^3dy+4\Pbb{X>A}A^3\]
and thus
\begin{align*}
&\Ebb{\langlerangle{Y^n_s,\phi}^2\langlerangle{X^n_s,1}^2;\labsrabs{\langlerangle{X^n_s,1}}>A}\\ 
&\leq 8n\lbrb{1+\langlerangle{X_s,1}^2}\lbrb{\int_{A}^{\infty}\Pbb{\langlerangle{X^n_s,1}>y}y^3dy+A^3\Pbb{\langlerangle{X^n_s,1}>A}}
\end{align*}
Now for $s\geq 0$ fixed and $n$ we use Lemma \ref{lem:largedeviation} with $y=2(e^{s+t(y)}-1)$, for any $y\geq A$ and $\beta\in(0,1)$ to get
 \begin{align*}
 &\Ebb{\langlerangle{Y^n_s,\phi}^2\langlerangle{X^n_s,1}^2;\labsrabs{\langlerangle{X^n_s,1}}>A} \\
 &\leq 8n\lbrb{1+\langlerangle{X_s,1}^2}\lbrb{\int_{A}^{\infty} e^{-n(1-\beta)t(y)}y^3dy+\lbrb{2(e^{s+t(A)}-1)}^3e^{-n(1-\beta)t(A)}} \\
 &+ 8n\lbrb{1+\langlerangle{X_s,1}^2}\lbrb{\int_{A}^{\infty}\Pbb{\langlerangle{X^0_n,1}>[\beta y]}y^3dy}\\
 &+ 8n\lbrb{1+\langlerangle{X_s,1}^2}A^3\Pbb{\langlerangle{X^0_n,1}>[\beta A]},
  \end{align*}
  where for convenience we either keep $A$ or substitute it with  $2(e^{s+t(A)}-1)$.
  
  First note that according to Lemma \ref{lem:largedeviation} we have that for $A$ big enough and $y>A$ then $t(y)\sim \ln(y)$ uniformly for $s\in[0,T]$ and  $t(y)\geq t(A)\geq \ln(1+A/2)-T$ which shows that
  \begin{align*}
  &\sup_{n\geq 1}\sup_{s\leq T}\lbrb{ 8n\lbrb{1+\langlerangle{X_s,1}^2}\lbrb{\int_{A}^{\infty} e^{-n(1-\beta)t(y)}y^3dy+\lbrb{2(e^{s+t(A)}-1)}^3e^{-n(1-\beta)t(A)}}}<\infty
  \end{align*}
  Clearly, since $X_s$ is deterministic we have that for $A>2\langlerangle{X_0,1}$ and $y\geq A$ 
  \[\Pbb{\langlerangle{X^n_0,1}>y}\leq \Pbb{\langlerangle{X^n_0-X_0,1}>y/2}\leq \frac{2^{4+\delta}\Ebb{\labsrabs{\langlerangle{X^n_0-X_0,1}}^{4+\delta}}}{y^{4+\delta}}\wedge 1.\]
  Thus using this last inequality, for $A$ big enough, we have that 
   \begin{align*}
    &\sup_{n\geq 1}\sup_{s\leq T}\left\{8n\lbrb{1+\langlerangle{X_s,1}^2}\lbrb{\int_{A}^{\infty}\Pbb{\langlerangle{X^n_0,1}>[\beta y]}y^3dy}\right\}\\
    &+\sup_{n\geq 1}\sup_{s\leq T}\left\{8n\lbrb{1+\langlerangle{X_s,1}^2}A^3\Pbb{\langlerangle{X^n_0,1}>[\beta A]}\right\}<\infty
   \end{align*}
   since by \eqref{moment condition} and the definition of $Y_0^n$,
   \[
   \sup_{n\geq 1}\lbrb{n\Ebb{\labsrabs{\langlerangle{X^n_0-X_0,1}}^{4+\delta}}}\leq \sup\limits_{n\geq1}\lbrb{\sup\limits_{{\|\phi\|}_{\infty}\leq 1}\E\labsrabs{\langle Y^n_0,\phi\rangle}^{4+\delta}}<\infty.
   \]
\end{proof}
The next result is the key to the estimates.
\begin{lem}\label{lem:estimateInteraction}
If for some $m\in\Nb^+$, $\alpha(x,y)=\sum_{i=1}^{m^d}f_i(x)g_i(y)$, where $f_i,g_i\in C^{\infty}(\Xc)$ and $f_i\geq 0$, $g_i\geq 0$ for all $1\leq i\leq m^d$ then 
\begin{eqnarray}\label{eq:esimateInteraction}
&&\sup_{\phi\in\Bc(\X); \|\phi\|_{\infty}\leq 1}\Ebb{\lbrb{\int_{\Xc}\phi(x)Y^n_s(dx)\int_{\Xc}\alpha(x,y)X^n_s(dy)}^2}\nonumber\\
&& \leq Cm^{d}\sup_{\phi\in\Bc(\X); \|\phi\|_{\infty}\leq 1} \Ebb{\langlerangle{\phi,Y^n_s}^2\langlerangle{1,X^n_s}^2}
\end{eqnarray}
\end{lem}
\begin{proof}
The proof uses the inequality $(a_1+\ldots+a_{m^d})^2\leq m^d\lbrb{a^2_1+\ldots+a^2_{m^d}}$ which readily yields
\begin{eqnarray*}
&&\sup_{||\phi||_\infty\leq 1}\Ebb{\lbrb{\int_{\Xc}\phi(x)Y^n_s(dx)\int_{\Xc}\alpha(x,y)X^n_s(dy)}^2}\\
&&\leq \sup_{||\phi||_\infty\leq 1}\Ebb{\lbrb{\sum_{i=1}^{m^d}\langlerangle{\phi f_i,Y^n_s}\langlerangle{g_i,X^s_n}}^2}\\
&& = FG\sup_{||\phi||_\infty\leq 1}\Ebb{\lbrb{\sum_{i=1}^m\langlerangle{\phi f_i/F,Y^n_s}\langlerangle{g_i/G,X^s_n}}^2}\\
&&\leq m^{d} FG \sup_{||\phi||_\infty\leq 1}\Ebb{\langlerangle{\phi,Y^n_s}^2\langlerangle{1,X^n_s}^2},
\end{eqnarray*}
where $F:=\max_{i\leq m^d}||f_i||_{\infty}$ and $G:=\max_{j\leq m^d}||g_j||_{\infty}$ and we have made use of the fact that $0\leq g_j/G\leq 1$, $X^n_s$ is a positive random measure and $\phi f_i/F\in \Bc(\X)$. This proves \eqref{eq:esimateInteraction}.
\end{proof}

\begin{prop}\label{Tightness}
  Consider a sequence of processes $(Y_t^n)_{t\geq 0}$ in $\D([0,\infty),\M_F(\X))$ and $Y_0^n$ satisfying \eqref{moment condition}.
Then, for any $\phi\in \Bc(\X)$,
the sequence of laws of the processes $\{\langle Y_{\cdot}^n, \phi\rangle;n\geq1\}$ is tight in $\D([0,\infty),\R)$.
\end{prop}
\begin{proof}  Since $\{\langle Y_{\cdot}^n, \phi\rangle;n\geq1\}$ is a sequence of semimartingales, we verify the tightness criteria given by Aldous \cite{Aldous} and Rebolledo (see, e.g., Etheridge\cite[Theorem 1.17]{Etheridge_introductory}).\\
For any fixed $t>0$,\, $\{\langle Y_t^n, \phi\rangle;n\geq1\}$ is tight due to Lemma \ref{Moment estimate Y_lemma}.\\
To the end, we will prove the tightness criterion of the finite variation part (say $A_t^n$) and the quadratic variation of the martingale part $M_t^n(\phi)$ of $\{\langle Y_{\cdot}^n, \phi\rangle;n\geq1\}$, respectively.\\
For any $\varepsilon>0$ and $T>0$, given a sequence of stopping time $\tau_n$ bounded by T. W.O.L.G., assume $\|\phi\|_{\infty}\leq 1$. As for the finite variation part $A_t^n$ of \eqref{martingale BPDL rescaled_Y^n}, we have
\beq\label{tightness drift part}
\bea
\sup_{n\geq1}\sup_{\theta\in[0,\de]}\PP\Big[\Big|A_{\tau_n+\theta}^{(n)}-&A_{\tau_n}^{(n)}\Big|>\varepsilon\Big]\\
&\leq\frac{1}{\varepsilon^2}\sup\limits_{n\geq1}\sup\limits_{\theta\in[0,\de]}\E\Big(A_{\tau_n+\theta}^{(n)}-A_{\tau_n}^{(n)}\Big)^2\\
&\stackrel{\text{H}\ddot{o}\text{lder}}{\leq}\frac{\delta}{\varepsilon^2}\sup_{n\geq1}\sup_{\theta\in[0,\de]}\int_{\tau_n}^{\tau_n+\theta}
\E\Big\{\int_{\X}Y_s^n(dx)b(x)\int_{\R^d}\phi(x+z)D(x,dz)\\
&\qquad-\int_{\X}Y_s^n(dx)d(x)\phi(x)\\
&\qquad-\int_{\X}Y_s^n(dx)\phi(x)\int_{\X}\al(x,y)X_s^n(dy)\\
&\qquad-\int_{\X}X_s(dx)\phi(x)\int_{\X}\al(x,y)Y_s^n(dy)\Big\}^2ds.
\eea
\eeq
We use the same estimates as before and apply Lemma \ref{lem:estimateInteraction} to the term before the last to easily get that
\beq
\bea
&\leq\frac{2\de\bar{b}^2}{\varepsilon^2}\cdot\sup\limits_{n\geq1}\int_0^T\E\sup\limits_{0\leq u\leq T}\langle Y_u^n,  ,\frac{ b(\cdot)}{\bar{b}}\int_{\X}\phi(\cdot+z)D(\cdot, dz)\rangle^2 ds\\
&\qquad+\frac{2\de\bar{d}^2}{\varepsilon^2}\cdot\sup\limits_{n\geq1}\int_0^T\E\sup\limits_{0\leq u\leq T}\langle Y_u^n,  \frac{d(\cdot)}{\bar{d}}\phi(\cdot)\rangle^2 ds\\
&\qquad+\frac{2\de\bar{\al}^2C}{\varepsilon^2}\cdot\lbrb{\sup\limits_{n\geq1}\sup_{||\phi||_\infty\leq 1}\int_0^T\E\sup\limits_{0\leq u\leq T}\langle Y_u^n,  \phi\rangle^2 ds+D}\\
&\leq \de TC_T^{(2)}C,
\eea
\eeq
where $C$ changes from line to line.
On the other hand, from \eqref{martingale BPDL rescaled_variation}, we have that
\beq\label{tightness martingale part}
\bea
&\sup_{n\geq1}\sup_{\theta\in[0,\de]}\PP\Big[\Big|\langle M_{\cdot}^n(\phi)\rangle_{\tau_n+\theta}-\langle M_{\cdot}^n(\phi)\rangle_{\tau_n}\Big|>\varepsilon\Big]\\
&\leq \frac{\de(\bar{b}+\bar{d})}{\varepsilon}\cdot\sup\limits_{n\geq1}\E\sup\limits_{0\leq u\leq T}\langle X_u^n,  1\rangle ds\\&\qquad+\frac{\de\bar{\al}}{\varepsilon}\cdot\sup\limits_{n\geq1}\E\sup\limits_{0\leq u\leq T}\langle X_s^n,  1\rangle^2 ds\\
&\leq \de (C_T^{(0)}+C_T^{(1)})C.
\eea
\eeq
According to the moment estimates results in Lemma \ref{Moment estimate X_lemma} and Lemma \ref{Moment estimate Y_lemma}, both inequalities \eqref{tightness drift part} and \eqref{tightness martingale part} can be less than $\varepsilon$ if we take $\de$ (which only depends on $T, \varepsilon, \|\phi\|_{\infty}$) small enough, i.e.
\[
\bea
\sup_{n\geq1}\sup_{\theta\in[0,\de]}\PP\Big[\Big|A_{\tau_n+\theta}^{(n)}-&A_{\tau_n}^{(n)}\Big|>\varepsilon\Big]<\varepsilon,\\
\sup_{n\geq1}\sup_{\theta\in[0,\de]}\PP\Big[\Big|\langle M_{\cdot}^n(\phi)\rangle_{\tau_n+\theta}-&\langle M_{\cdot}^n(\phi)\rangle_{\tau_n}\Big|>\varepsilon\Big]<\varepsilon,
\eea
\]
which fulfil the Aldous-Rebolledo tightness condition.
\end{proof}

\subsection{Convergence of finite dimensional distributions}
In this section, we prove a weak limit  of $\{(Y_t^n)_{t\geq 0};n\geq1\}$ in the sense of f.d.d. convergence is a solution of some martingale problem.

\begin{prop}\label{F.D.D limit}
  Under the conditions given in Theorem $\ref{Theorem CLT}$, the finite dimensional distributions of $(Y_t^n)_{t\geq 0}$ converge as $n\to\infty$ to those of a $\mathcal{S}'(\X)$-valued Markov process $(Y_t)_{t\geq0}$ satisfying that for $\phi\in\mathcal{S}(\R^d)$, the process
 \beq\label{f.d.d limit_martingale}
\bea
M_t(\phi):=&\langle Y_t,\phi\rangle-\langle \gamma,\phi\rangle-\int_0^t\Big\langle Y_s,\, b(\cdot)\int_{\R^d}\phi(\cdot+z)D(\cdot,dz)\Big\rangle ds\\
&+\int_0^t\Big\langle Y_s,\, d(\cdot)\phi(\cdot)\Big\rangle ds\\
&+\int_0^t\Big\langle Y_s,\, \int_{\X}\al(x,\cdot)\phi(x)X_s(dx)\Big\rangle ds\\
&+\int_0^t\Big\langle Y_s, \, \phi(\cdot)\int_{\X}\al(\cdot,y)X_s(dy)\Big\rangle ds\\.
\eea
 \eeq
 is a continuous martingale with quadratic variation
\beq\label{f.d.d limit_martingale quadratic variation}
\bea
\langle M_{\cdot}(\phi)\rangle_t=&\int_o^tds\int_{\X}X_s(dx)b(x)\int_{\R^d}\phi^2(x+z)D(x,dz)\\
                        &+\int_0^tds\int_{\X}X_s(dx)\phi^2(x)\Big(d(x)+\int_{\X}\al(x,y)X_s(dy)\Big).
\eea
\eeq
\end{prop}

\begin{proof}  By Proposition \ref{Tightness}, we already proved $\{\langle Y_{\cdot}^n, \phi\rangle;n\geq1\}$ is tight in $\D([0,\infty),\R)$ for any $\phi\in\mathcal{S}(\R^d)$. Following Mitoma \cite{Mitoma} (see e.g., Ethier and Kurtz \cite[Theorem 3.9.1]{Ethier_Kurtz}), we conclude that the sequence $\{(Y_t^n)_{t\geq 0};n\geq1\}$ is tight in $\D([0,\infty),\mathcal{S}'(\R^d))$.
Hence, we can assume there exists a weak limit $(Y_t)_{t\geq0}$ of a subsequence of  $\{(Y_t^n)_{t\geq 0}; n\geq1\}$.
Since $Y^n_t\in\M_F(\X)$, then $\langlerangle{Y^n_t, \phi}=0$ for any $\phi\in\mathcal{S}(\X^c)$. Therefore, we have $Y_t\in \mathcal{S}'(\X)$.  

Firstly, we check that $(Y_t)_{t\geq 0}$ is a.s. continuous. By the construction of $(Y_t^n)$, we have
\beq
\bea
\sup\limits_{t\in[0,T]}\sup\limits_{\|\phi\|\leq 1}|\langle Y_t^n, \phi\rangle-\langle Y_{t-}^n, \phi\rangle|&\leq\sup\limits_{t\in[0,T]}\sup\limits_{\|\phi\|\leq 1}\sqrt{n}\big\{|\langle X_t^n, \phi\rangle-\langle X_{t-}^n, \phi\rangle|+|\langle X_t-X_{t-}, \phi\rangle|\big\}\\
&\leq\sqrt{n}\frac{1}{n}+0\\
&=\frac{1}{\sqrt{n}}.
\eea
\eeq
By letting $n\to\infty$, it implies the continuity of $(Y_t)_{t\geq0}$, i.e. $(Y_t)_{t\geq0}\in\mathbb{C}([0,+\infty), \mathcal{S}'(X))$.

To prove ($M_t(\phi))_{t\geq0}$ is a martingale, it suffices to prove that
\beq
\E [M_t(\phi)] =0.
\eeq
Let
\beq
\bea
\widetilde{M}_t^n(\phi):=&\langle Y_t^n,\phi\rangle-\langle Y_0^n,\phi\rangle-\int_0^tds\int_{\X}Y_s^n(dx)b(x)\int_{\R^d}\phi(x+z)D(x,dz)\\
&+\int_0^tds\int_{\X}d(x)\phi(x)Y_s^n(dx)\\
&+\int_0^tds\int_{\X}X_s(dx)\phi(x)\int_{\X}\al(x,y)Y_s^n(dy)\\
&+\int_0^tds\int_{\X}Y_s^n(dx)\phi(x)\int_{\X}\al(x,y)X_s(dy).
\eea
 \eeq
Then, for fixed $t>0$ and any $n\in\N$, we have
 \beq\label{f.d.d limit_martingale proof}
 |\E [M_t(\phi)]|\leq |\E [M_t^n(\phi)-\widetilde{M}_t^n(\phi)]|+|\E [\widetilde{M}_t^n(\phi)-M_t(\phi)]|+|\E [M_t^n(\phi)]|.
 \eeq
According to Proposition \ref{Martingale BPDL Rescaled}, we have $\E [M_t^n(\phi)]=0$. \\
Since $\{(Y_t^n)_{t\geq 0};n\geq1\}$ converges in law to $(Y_t)_{t\geq0}$ as $n\to\infty$ and $\big(\widetilde{M}_t^n(\phi)-M_t(\phi)\big)$ is homogeneous w.r.t. $\big(Y_t^n-Y_t\big)$, we get
\begin{equation}
\lim\limits_{n\to\infty}|\E [\widetilde{M}_t^n(\phi)-M_t(\phi)]|=0.
\end{equation}
As for the first term on RHS of \eqref{f.d.d limit_martingale proof},
\beq
\bea
|\E [M_t^n(\phi)&-\widetilde{M}_t^n(\phi)]|\\
&=\Bigg |\E\Bigg\{\sqrt{n}\int_0^tds\int_{\X}\Big(\frac{Y_s^n(dx)}{\sqrt{n}}+X_s(dx)\Big)\phi(x)\int_{\X}\al(x,y)\Big(\frac{Y_s^n(dy)}{\sqrt{n}}+X_s(dy)\Big)\\
&\qquad-\sqrt{n}\int_0^tds\int_{\X}X_s(dx)\phi(x)\int_{\X}\al(x,y)X_s(dy)\\
&\qquad-\int_0^tds\int_{\X}X_s(dx)\phi(x)\int_{\X}\al(x,y)Y_s^n(dy)\\
&\qquad-\int_0^tds\int_{\X}Y_s^n(dx)\phi(x)\int_{\X}\al(x,y)X_s(dy)\Bigg\}\Bigg |\\
&\leq \frac{1}{\sqrt{n}}\Bigg|\E\int_0^tds\int_{\X}Y_s^n(dx)\phi(x)\int_{\X}\al(x,y)Y_s^n(dy)\Bigg|\\
&\leq\frac{1}{\sqrt{n}}\bar{\al}T\|\phi\|_{\infty}C_t^{(2)}\\
&\stackrel{n\to\infty}{\longrightarrow}0,
\eea
\eeq
where $C_t^{(2)}$ is determined as in Lemma \ref{Moment estimate Y_lemma}.\\
By combining the above estimates together, we conclude $|\E [M_t(\phi)]| =0.$

In the remainder, we will justify that the quadratic variation of $M_t(\phi)$ has the form \eqref{f.d.d limit_martingale quadratic variation}.\\
By applying It$\hat{o}$'s formula to $\langle Y_t,\phi\rangle^2$, according to the  semimartingale decomposition \eqref{f.d.d limit_martingale} of $\langle Y_t,\phi\rangle$, we have
\beq\label{f.d.d limit_variation proof_Ito of square}
\bea
\langle Y_t,\phi\rangle^2
&=\langle \gamma,\phi\rangle^2+2\int_0^t \langle Y_s,\phi\rangle d[\langle Y_s,\phi\rangle]+\langle M_.(\phi)\rangle_t\\
&=\langle \gamma,\phi\rangle^2+\langle M_.(\phi)\rangle_t\\
&\quad+2\int_0^t\langle Y_s,\phi\rangle ds\Big\{ \int_{\X}Y_s(dx)b(x)\int_{\R^d}\phi(x+z)D(x,dz)-\int_{\X}Y_s(dx)d(x)\phi(x)\\
&\quad-\int_{\X}Y_s(dx)\phi(x)\int_{\X}\al(x,y)X_s(dy)-\int_{\X}X_s(dx)\phi(x)\int_{\X}\al(x,y)Y_s(dy)\Big\}\\
&\quad+\text{martingale}.
\eea
\eeq

On the other hand, according to the definition of $(Y_t^n)$, we have

\beq\label{f.d.d limit_variation proof_square decomposition}
\bea
\langle Y_t^n, \phi\rangle^2&=\langle \sqrt{n}(X_t^n-X_t), \phi\rangle^2\\
&=n\Big[\langle X_t^n, \phi\rangle^2-2\langle X_t^n,\phi\rangle\langle X_t,\phi\rangle+\langle X_t,\phi\rangle^2\Big].
\eea
\eeq
To simplify the computations, let us introduce new notation:
\beq\label{f.d.d limit_variation proof_notation}
\bea
A(s)&:=\int_{\X}X_s(dx)\Big[b(x)\int_{\R^d}\phi(x+z)D(x,dz)-\phi(x)\big(d(x)+\int_{\X}\alpha(x,y)X_s(dy)\big)\Big],\\
B^n(s)&:=\int_{\X}X_s^n(dx)\Big[b(x)\int_{\R^d}\phi(x+z)D(x,dz)-\phi(x)\big(d(x)+\int_{\X}\alpha(x,y)X_s^n(dy)\big)\Big].
\eea
\eeq
From \eqref{martingale for X}, \eqref{martingale for X^2} and \eqref{LLN limit}, respectively, it follows that
\beq
\bea
\langle X_t^n,\phi\rangle&=\langle X_0^n,\phi\rangle+\int_0^t B^n(s)ds+\text{martingale},\\
\langle X_t^n,\phi\rangle^2&=\langle X_0^n,\phi\rangle^2+2\int_0^t\langle X_s^n,\phi\rangle B^n(s)ds+\frac{1}{n}\int_0^tds\int_{\X}X_s^n(dx)\\&\quad\Big[b(x)\int_{\R^d}\phi^2(x+z)D(x,dz)
+\phi^2(x)\big(d(x)+\int_{\X}\alpha(x,y)X_s^n(dy)\big)\Big]+\text{martingale},\\
\langle X_t,\phi\rangle&=\langle X_0,\phi\rangle+\int_0^t A(s)ds.
\eea
\eeq
By substituting every term above into \eqref{f.d.d limit_variation proof_square decomposition}, we have that
\beq\label{f.d.d limit_variation proof_fist formula}
\bea
\langle Y_t^n,\phi\rangle^2&=n\langle X_0^n,\phi\rangle^2+\int_0^tds\int_{\X}X_s^n(dx)\Big[b(x)\int_{\R^d}\phi^2(x+z)D(x,dz)\\
&\quad+\phi^2(x)\big(d(x)+\int_{\X}\alpha(x,y)X_s^n(dy)\big)\Big]+2n\int_0^t\langle X_s^n,\phi\rangle B^n(s)ds\\
&\quad-2n\Big[\langle X_0^n,\phi\rangle+\int_0^t B^n(s)ds\Big]\Big[\langle X_0,\phi\rangle+\int_0^t A(s)ds\Big]\\
&\quad +n\Big[\langle X_0,\phi\rangle+\int_0^t A(s)ds\Big]^2+\text{martingale}.
\eea
\eeq
Set
\beq
D^{t,n,1}:=\int_0^tds\int_{\X}X_s^n(dx)\Big[b(x)\int_{\R^d}\phi^2(x+z)D(x,dz)+\phi^2(x)\big(d(x)+\int_{\X}\alpha(x,y)X_s^n(dy)\big)\Big].
\eeq
By combining all the quadratic terms at time 0 in \eqref{f.d.d limit_variation proof_fist formula} together, it follows that
\beq
\bea
\eqref{f.d.d limit_variation proof_fist formula}=&n\langle X_0^n-X_0, \phi\rangle^2+D^{t,n,1}\\
&+2n\int_0^t\langle X_s^n, \phi\rangle B^n(s)ds-2n\langle X_t, \phi\rangle\int_0^t B^n(s)ds\\
&-2n\langle X_0^n, \phi\rangle\int_0^tA(s)ds+2n\langle X_0, \phi\rangle\int_0^tA(s)ds\\
&+n\Big[\int_0^tA(s)ds\Big]^2+\text{martingale}
\eea
\eeq
\beq
\bea
=&\langle Y_0^n, \phi\rangle^2+D^{t,n,1}\\
&+2n\int_0^t\langle \frac{Y^n_s}{\sqrt{n}}+X_s, \phi\rangle B^n(s)ds-2n\langle X_t, \phi\rangle\int_0^t B^n(s)ds\\
&-2\sqrt{n}\langle Y_0^n, \phi\rangle\int_0^tA(s)ds+n\Big[\int_0^tA(s)ds\Big]^2+\text{martingale}
\eea
\eeq
\beq
\bea
=&\langle Y_0^n, \phi\rangle^2+D^{t,n,1}\\
&+2\sqrt{n}\int_0^t\langle Y^n_s, \phi\rangle B^n(s)ds\\
&+2n\int_0^t\langle X_s, \phi\rangle B^n(s)ds-2n\langle X_t, \phi\rangle\int_0^t B^n(s)ds\\
\quad&-2\sqrt{n}\langle Y_0^n, \phi\rangle\int_0^tA(s)ds+n\Big[\int_0^tA(s)ds\Big]^2+\text{martingale}
\eea
\eeq
\beq
\bea
\stackrel{\text{Integration by parts}}{=}&\langle Y_0^n, \phi\rangle^2+D^{t,n,1}\\
&+2\sqrt{n}\int_0^t\langle Y^n_s, \phi\rangle B^n(s)ds\\
&-2n\int_0^tdsA(s)\int_0^s B^n(r)dr-2\sqrt{n}\langle Y_0^n, \phi\rangle\int_0^tA(s)ds\\
&+n\Big[\int_0^tA(s)ds\Big]^2+\text{martingale}
\eea
\eeq
\beq\label{f.d.d limit_variation proof_middle formula}
\bea
\stackrel{\text{Replace}\, B^n(s)\, \text{by} \eqref{f.d.d limit_variation proof_notation}}{=}&\langle Y_0^n, \phi\rangle^2+D^{t,n,1}\\
&+2\int_0^t\langle Y_s^n,\phi\rangle ds\Big\{ \int_{\X}Y_s^n(dx)\Big[b(x)\int_{\R^d}\phi(x+z)D(x,dz)-\phi(x)\Big(d(x)\\
&\quad+\int_{\X}\al(x,y)X_s^n(dy)\Big)\Big]-\int_{\X}X_s(dx)\phi(x)\int_{\X}\al(x,y)Y_s^n(dy)\Big\}\\
&+2\sqrt{n}\int_0^t\langle Y^n_s, \phi\rangle A(s)ds\\
&-2n\int_0^tdsA(s)\int_0^s B^n(r)dr-2\sqrt{n}\langle Y_0^n, \phi\rangle\int_0^tA(s)ds\\
&+n\Big[\int_0^tA(s)ds\Big]^2+\text{martingale}.
\eea
\eeq
Set
\beq
\bea
D^{t,n,2}&:=2\int_0^t\langle Y_s^n,\phi\rangle ds\Big\{ \int_{\X}Y_s^n(dx)\Big[b(x)\int_{\R^d}\phi(x+z)D(x,dz)-d(x)\phi(x)\\
&\quad-\phi(x)\int_{\X}\al(x,y)X_s^n(dy)\Big]-\int_{\X}X_s(dx)\phi(x)\int_{\X}\al(x,y)Y_s^n(dy)\Big\}.\\
\eea
\eeq
Replacing $\langle Y^n_s, \phi\rangle$ and $B^n(r)$ by \eqref{martingale BPDL rescaled_Y^n} and \eqref{f.d.d limit_variation proof_notation} respectively, one obtains that
\beq
\bea
\eqref{f.d.d limit_variation proof_middle formula}
=&\langle Y_0^n, \phi\rangle^2+D^{t,n,1}+D^{t,n,2}\\
&+2\sqrt{n}\int_0^tdsA(s)\int_0^sdr\int_{\X}Y_r^n(dx)\Big[b(x)\int_{\R^d}\phi(x+z)D(x,dz)-d(x)\phi(x)\Big]\\
&-2n\int_0^tdsA(s)\int_0^sdr\int_{\X}X_r^n(dx)\phi(x)\int_{\X}\al(x,y)X_r^n(dy)\\
&+2n\int_0^tdsA(s)\int_0^sdr\int_{\X}X_r(dx)\phi(x)\int_{\X}\al(x,y)X_r(dy)\\
&-2n\int_0^tdsA(s)\int_0^sdr\int_{\X}X_r^n(dx)\Big[b(x)\int_{\R^d}\phi(x+z)D(x,dz)\\
&\quad-\phi(x)\big(d(x)+\al(x,y)X_r^n(dy)\big)\Big]
+n\Big[\int_0^tA(s)ds\Big]^2+\text{martingale}
\eea
\eeq
\beq
\bea
=&\langle Y_0^n, \phi\rangle^2+D^{t,n,1}+D^{t,n,2}\\
&+2\sqrt{n}\int_0^tdsA(s)\int_0^sdr\int_{\X}Y_r^n(dx)\Big[b(x)\int_{\R^d}\phi(x+z)D(x,dz)-d(x)\phi(x)\Big]\\
&+2n\int_0^tdsA(s)\int_0^sdr\int_{\X}X_r(dx)\phi(x)\int_{\X}\al(x,y)X_r(dy)\\
&-2n\int_0^tdsA(s)\int_0^sdr\int_{\X}X_r^n(dx)\Big[b(x)\int_{\R^d}\phi(x+z)D(x,dz)-d(x)\phi(x)\Big]\\
&+n\Big[\int_0^tA(s)ds\Big]^2+\text{martingale}.
\eea
\eeq
Recombining $X_r^n,\,X_r$ and $Y_r^n$, it thus follows
\beq
\bea
\langle Y_t^n, \phi\rangle^2=&\langle Y_0^n, \phi\rangle^2+D^{t,n,1}+D^{t,n,2}\\
&-2n\int_0^tdsA(s)\int_0^sA(r)dr\\
&+n\Big[\int_0^tA(s)ds\Big]^2+\text{martingale}\\
\stackrel{\text{Integration by parts}}{=}&\langle Y_0^n, \phi\rangle^2+D^{t,n,1}+D^{t,n,2}+\text{martingale}.
\eea
\eeq
Obviously, both $D^{t,n,1}$ and $D^{t,n,2}$ converge as $n\to\infty$.

Finally, we get that
\beq\label{f.d.d limit_variation proof_limit of square}
\bea
\langle Y_t,\phi\rangle^2
&=\langle Y_0,\phi\rangle^2+\int_o^tds\int_{\X}X_s(dx)b(x)\int_{\R^d}\phi^2(x+z)D(x,dz)\\
&\qquad+\int_0^tds\int_{\X}X_s(dx)\phi^2(x)\Big(d(x)+\int_{\X}\al(x,y)X_s(dy)\Big)\\
&\qquad+2\int_0^t\langle Y_s,\phi\rangle ds\Big\{ \int_{\X}Y_s(dx)\Big[b(x)\int_{\R^d}\phi(x+z)D(x,dz)-d(x)\phi(x)\\
&\qquad-\phi(x)\int_{\X}\al(x,y)X_s(dy)\Big]-\int_{\X}X_s(dx)\phi(x)\int_{\X}\al(x,y)Y_s(dy)\Big\}\\
&\qquad+\text{martingale}.
\eea
\eeq
By comparing the representations of \eqref{f.d.d limit_variation proof_Ito of square} and \eqref{f.d.d limit_variation proof_limit of square}, we conclude that
\beq
\bea
\langle M_.(\phi)\rangle_t&=\int_o^tds\int_{\X}X_s(dx)b(x)\int_{\R^d}\phi^2(x+z)D(x,dz)\\
&\qquad+\int_0^tds\int_{\X}X_s(dx)\phi^2(x)\Big(d(x)+\int_{\X}\al(x,y)X_s(dy)\Big).
\eea
\eeq
\end{proof}
\subsection{Uniqueness of the martingale problem}
Instead of proving the uniqueness of the solution of the limiting martingale problem directly, we associate it with the solution of a corresponding generalized Langevin equation which will be shown in the next section. We will show that the Langevin equation has an unique solution (see Theorem \ref{Langevin equation}).

\section{Links with generalized Langevin equations}\label{section four}
A criterion for an infinite-dimensional Gaussian process (distribution-valued process) to satisfy a generalized Langevin equation is given in \cite{Boj_Gor}, where both of the evolution term and the white noise term are time inhomogeneous. In this section, we apply the criterion to our fluctuation limit obtained in the previous section.

\begin{Def}
An $\mathcal{S}'(\R^d)$-valued process $\{W_t; t\in\R^+\}$ is called (centered) Gaussian if $\{\langle W_t,\phi\rangle; t\in \R^+,\phi\in\mathcal{S}(\R^d)\}$ is a (centered) Gaussian system. We say that $\{W_t; t\in\R^+\}$ is a $\mathcal{S}'(\X)$-valued process if $\langle W_t,\phi\rangle\equiv 0$, for any $\phi\in\mathcal{S}(\X^c)$. 
\end{Def}
\begin{Def}\label{Wiener process_Def}
A centered Gaussian $\mathcal{S}'(\R^d)$-valued process $W=\{W_t; t\in\R^+\}$ is called a generalized Wiener process if it has continuous paths and its covariance functional $C(s,\phi;t,\psi):=\E[\langle W_s, \phi\rangle\langle W_t,\phi\rangle]$ has the form
\beq\label{eq:covariance} C(s,\phi;t,\psi)=\int_0^{s\wedge t}\langle Q_u\phi,\psi\rangle du, \qquad\qquad s,t\in \R^+, \phi,\psi\in\mathcal{S}(\R^d),
\eeq
where the operators $Q_u: \mathcal{S}(\R^d)\to\mathcal{S}'(\R^d)$ have the following properties:
\begin{enumerate}
  \item  $Q_u$ is linear, continuous, symmetric and positive for each $u\in\R^+$,
  \item  the function $u\to\langle Q_u\phi,\psi\rangle$ is right continuous with left limit for each $\phi,\psi\in\mathcal{S}(\R^d)$.
\end{enumerate}
We then say that $W$ is associated to $Q$.
\end{Def}
Let's remind that we inherit the same notation as in Section \ref{section two} and Section \ref{section three}.\\
Define $Q_t\phi\in\mathcal{S}'(\R^d)$ for any $\phi\in\mathcal{S}(\R^d)$ and $t\in\R^+$ by
\beq\label{langevin equation_covariance measure}
\bea
\langle Q_t\phi,\psi\rangle:=\int_{\X}X_t(dx)&\Bigg[b(x)\int_{\Xc-x}\phi(x+z)\psi(x+z)D(x,dz)\\
                        &+\phi(x)\psi(x)\Big(d(x)+\int_{\X}\al(x,y)X_t(dy)\Big)\Bigg], \qquad \text{for}~ \psi\in\mathcal{S}(\R^d).
\eea
\eeq
Recall the quadratic variation form of $M_t(\phi)$ in \eqref{CLT variation}. It follows from a direct fact that $\langle M_.(\phi)\rangle_t=\int_0^t\langle Q_u\phi,\phi\rangle du$.
Then, we have
\begin{teor}\label{Langevin equation}
The fluctuation limit process $(Y_t)_{t\geq0}$ obtained in Theorem \ref{Theorem CLT} is the unique solution of a time inhomogeneous Langevin equation
\beq\label{langevin equation_infinite dim}
\left\{
  \begin{array}{ll}
    dY_t= A_t^*Y_tdt+dW_t, & \qquad t>0 \\
    Y_0 = \gamma &
  \end{array}
\right.,
\eeq
where $A_t^*$ denotes the adjoint operator of $A_t$ defined by
\beq\label{langevin equation_drift}
\bea
A_t\phi(x)=b(x)\int_{\Xc}&\phi(x+z)D(x,dz)-\phi(x)\big(d(x)+\int_\X\al(x,y)X_t(dy)\big)\\
&-\int_{\X}\al(y,x)\phi(y)X_t(dy),
\eea
\eeq
and $(W_t)_{t\geq0}$ is an $\mathcal{S}'(\X)$-valued Wiener process with covariance
\beq\label{langevin equation_diffusion}
\E\big[\langle W_s,\phi\rangle\langle W_t,\psi\rangle\big]=\int_0^{s\wedge t}\langle Q_u\phi,\psi\rangle du, \qquad s,t\geq0,~\phi,\psi\in\mathcal{S}(\R^d).
\eeq
\end{teor}
\begin{rem}
\begin{enumerate}
  \item An $\mathcal{S}'(\R^d)$-valued process $(Y_t)_{t\geq0}$ is said to be a solution of $\eqref{langevin equation_infinite dim}$ if for each $\phi\in\mathcal{S}(\R^d)$,
      \beq\label{langevin equation_weak form}
      \langle Y_t,\phi\rangle=\langle \gamma,\phi\rangle+\int_0^t\langle Y_u, A_u\phi\rangle du+\langle W_t,\phi\rangle, \qquad \text{for}~ t\in \R^+.
      \eeq
  \item $(W_t)_{t\geq0}$ has independent increments but not the stationary property since the covariance functional $Q$ depends on the time.
\end{enumerate}
\end{rem}
\begin{proof}
 {\bf Existence.} According to Theorem \ref{Theorem CLT}, the covariance functional of the continuous martingale $M_t$ on testing functions is deterministic, which implies that $(M_t)_{t\geq0}$ is a $\mathcal{S}'(\R^d)$-valued mean zero Gaussian process (see Walsh\cite[Proposition 2.10]{Walsh}). Hence, $(Y_t)_{t\geq0}$ is also an $\mathcal{S}'(\R^d)$-valued Gaussian process. \\
Set $K(s,\phi;t,\psi):=\E[\langle Y_s,\phi\rangle\langle Y_t,\psi\rangle]$.
To the end, one needs eventually to show that
\beq
\frac{\partial}{\partial t}K(t,\phi;t,\psi)-K(t,A_t\phi;t,\psi)-K(t,\phi;t,A_t\psi)=\langle Q_t\phi,\psi\rangle.
\eeq
By applying It$\hat{o}$ formula to $\langle Y_s,\phi\rangle\langle Y_t,\psi\rangle$ and \eqref{CLT limit}, we have that
\beq
\langle Y_t, \phi\rangle\langle Y_t, \psi\rangle=\int_0^t\langle Y_u, \phi\rangle d\,\langle Y_u, \psi\rangle + \int_0^t\langle Y_u, \psi\rangle d\,\langle Y_u, \phi\rangle  +\int_0^t  d\,\langle Y_u, \psi\rangle d\,\langle Y_u, \phi\rangle.
\eeq
Then taking expectation on both sides, we obtain that
\beq
K(t,\phi; t,\psi)= \int_0^t\langle Y_u, \phi\rangle \langle Y_u, A_u \psi\rangle du + \int_0^t\langle Y_u, \psi\rangle \langle Y_u, A_u \phi\rangle du+\int_0^t d[M_.(\psi), M_.(\phi)]_u.
\eeq
Differentiate the last equation with respect to $t$, we conclude
\beq
\bea
\frac{\partial}{\partial t}K(t,\phi;t,\psi)
&=K(t,A_t\phi;t,\psi)+K(t,\phi;t,A_t\psi)+\frac{\partial}{\partial t}\E [M_t(\psi)M_t(\phi)]\\
&=K(t,A_t\phi;t,\psi)+K(t,\phi;t,A_t\psi)+\langle Q_t\phi,\psi\rangle,
\eea
\eeq
where the last equality is due to \eqref{CLT variation} and \eqref{langevin equation_covariance measure}.\\
On the other hand, it is not hard to check that $(Q_t)_{t\geq0}$ satisfies the conditions required in Definition \ref{Wiener process_Def}.
Finally, by the results of \cite[Theorem 2]{Boj_Gor}, there exists an $\mathcal{S}'(\R^d)$-valued Wiener process $(W_t)_{t\geq0}$ associated to the covariance functional $(Q_t)_{t\geq0}$ such that
$(Y_t)_{t\geq0}$ satisfies the generalized Langevin equation \eqref{langevin equation_infinite dim} driven by a generalized Wiener process $(W_t)_{t\geq0}$. It remains to show that moreover $(W_t)_{t\geq0}$ is in fact $\mathcal{S}'(\X)$-valued problem. Note that from \eqref{langevin equation_covariance measure} we have that if $\phi\in\mathcal{S}(\X^c)$ or $\psi\in\mathcal{S}(\X^c)$ then $\langlerangle{Q_t\phi,\psi}=0$, for any $t\geq 0$. This is due to the definition of $b(x)$, $d(x)$, $\alpha(x,y)$ and $D(x,dz)$. Thus in \eqref{eq:covariance} $C(s,\phi;s,\phi)=\Ebb{\langlerangle{W_s,\phi}^2}=0$ and therefore since $(W_t)_{t\geq 0}$ is centered we conclude that $\langlerangle{W_s,\phi}\equiv 0$.

{\bf Uniqueness.} First note that Assumption (A1) implies easily that $A_t\phi \in \mathcal{S}(\R^d)$ for any $\phi\in \mathcal{S}(R^d)$ since $\X$ and henceforth $\X-\X$ are compact and any differentiation of \eqref{langevin equation_drift} can be taken under the integrals of the right-hand side of \eqref{langevin equation_drift}. Indeed all terms but $\int_{\Xc}\phi(x+z)m(x,z)dz$ are obvious. However since $m(x,z)\in C^\infty(\X\times(\X-\X))$ then $\sup_{x\in\X,z\in \X-\X}\labsrabs{\frac{d^n}{dx^n}m(x,z)}<\infty$ for any $n\geq 0$ we conclude that $\int_{\Xc}\phi(x+z)m(x,z)dz\in\mathcal{S}(\R^d)$.

Since all coefficients are bounded, the linear operator $A_t:\mathcal{S}(\R^d)\mapsto\mathcal{S}(\R^d)$  is a supremum norm uniformly bounded for $t\in[0,T]$ for any $T>0$. Therefore, the equation \eqref{langevin equation_infinite dim} has an unique $\mathcal{S}'(\R^d)$-valued solution given by the mild form:
  \beq
  Y_t=T^*_{0,t}\gamma+\int_0^t T^*_{r,t}d W_r,
  \eeq
  where $\{T_{r,t}: 0\leq r\leq t<+\infty\}$ is the unique reversed evolution system generated by $(A_t)_{t\geq0}$ and $T^*_{r,t}$ is its adjoint operator of $T_{r,t}$.
  We refer the reader to \cite[Theorem 2.1]{Kallianpur_Perez-Abreu} for details.
\end{proof}

\section{One dimensional time-inhomogeneous Ornstein-Uhlenbeck process}\label{section five}
In this section, we will study a degenerate case as an example of Theorem \ref{Langevin equation}. Consider the case when there is no spatial dispersal and all the individuals stay at the same position, i.e. $D(x,dz)=1_{\{z=0\}}$ in dispersal kernel \eqref{dispersal kernel}.

\begin{prop}
  Admit the same conditions as in Theorem \ref{Theorem CLT}. In particular, assume $X_t^n=\xi_t^n\de_x$  and $D(x,dz)=1_{\{z=0\}}$ in \eqref{dispersal kernel}. Then, $(\xi_t^n, \eta_t^n)_{t\geq0}$ converge in law to $(\xi_t, \eta_t)_{t\geq0}$ as $n\to\infty$ which satisfies the following equations:
  \beq\label{langevin equation_one dim}
  \left\{
    \begin{array}{ll}
      d\xi_t=\big(b(x)-d(x)-\al(x,x)\xi_t\big)\xi_tdt \\
      d\eta_t=\big(b(x)-d(x)-2\al(x,x)\xi_t\big)\eta_tdt+\sqrt{\big(b(x)+d(x)+\al(x,x)\xi_t\big)\xi_t}dB_t,
    \end{array}
  \right.
  \eeq
where $\eta_t^n:=\sqrt{n}(\xi_t^n-\xi_t)$.
\end{prop}
\begin{rem}

We can regard the system above as an inhomogeneous Ornstein-Uhlenbeck (OU) process living in a deterministic environment. We refer the reader to \cite[Theorem 11.2.3]{Ethier_Kurtz} for a general limiting prcocess defined by a one-dim inhomogeneous Langevin equation.
\end{rem}
\begin{proof}  Since $D(x,dz)=1_{\{z=0\}}$, by taking $\phi=1$ in \eqref{LLN limit}, we can easily show that there exists a process $(\xi_t)_{t\geq0}$ defined by $\xi_t:=\langle X_t,1\rangle$ solving the first equation in \eqref{langevin equation_one dim}.
Taking $\phi=1$, from \eqref{langevin equation_drift}, we have that
\beq\label{OU_drift}
\langle Y_t, A_t1\rangle =\big(b(x)-d(x)-2\al(x,x)\xi_t\big) \langle Y_t,1\rangle.
\eeq
From \eqref{langevin equation_covariance measure} and \eqref{langevin equation_diffusion}, we have that
\beq
\E\langle W_t,1\rangle^2=\int_0^t\big(b(x)+d(x)+\al(x,x)\xi_s\big)\xi_sds.
\eeq
Define
\beq
B_t=\int_0^t\Big[\big(b(x)+d(x)+\al(x,x)\xi_s\big)\xi_s\Big]^{-\frac{1}{2}}d\langle W_s,1\rangle.
\eeq
Then, we get its quadratic variation $\langle B\rangle_t=t$.  Thus, $(B_t)_{t\geq0}$ is a standard Brownian motion. Furthermore, we have
\beq\label{OU_diffusion}
d\langle W_t,1\rangle=\sqrt{\big(b(x)+d(x)+\al(x,x)\xi_t\big)\xi_t}\cdot dB_t.
\eeq
Let $\eta_t:=\langle Y_t,1\rangle$, by taking \eqref{OU_drift} and \eqref{OU_diffusion} back to \eqref{langevin equation_weak form} when $\phi=1$, the second equation in \eqref{langevin equation_one dim} follows.
\end{proof}

In the next result, we study the stationary distribution of the system \eqref{langevin equation_one dim}.
\begin{prop}\label{Langevin equation_one dim}
  Suppose the process $(\eta_t)_{t\geq0}$ is defined as in \eqref{langevin equation_one dim}. Then, it has a stationary distribution which is Gaussian $\mathcal{N}(0,\frac{b(x)}{\al(x,x)})$.
\end{prop}
\begin{rem}
When we consider the long term behavior, as long as $b(x)>d(x)$, it always has the same fluctuation no matter which value the death rate $d(x)$ takes.
\end{rem}
\begin{proof}
Let \[\theta_t:=-\big(b(x)-d(x)-2\al(x,x)\xi_t\big),\]
\[\sigma_t:=\sqrt{\big(b(x)+d(x)+\al(x,x)\xi_t\big)\xi_t}.\]
From \eqref{langevin equation_one dim}, it follows that
\beq
d\eta_t=-\theta_t\eta_tdt+\sigma_tdB_t.
\eeq
The characteristic function of $(\eta_t)_{t\geq0}$ has the form
\beq
\E_{\eta_0}\big[e^{iz\eta_t}\big]=\exp{\Big\{ize^{-\int_0^t\theta_udu}\eta_0
-\frac{1}{2}z^2\int_0^t\sigma_u^2e^{-2\int_u^t\theta_vdv}du\Big\}}.
\eeq
Since $\xi_t$ in \eqref{langevin equation_one dim} has a unique stable equilibrium $\frac{b(x)-d(x)}{\al(x,x)}$, it follows that $\lim\limits_{t\to\infty}\theta_t=b(x)-d(x)>0$
and $\lim\limits_{t\to\infty}\sigma_t^2=2b(x)\big(b(x)-d(x)\big)/\al(x,x)$.

Then,
\beq
\bea
\lim\limits_{t\to\infty}\log\E_{\eta_0}\big[e^{iz\eta_t}\big]
&=-\lim\limits_{t\to\infty}\frac{1}{2}z^2\cdot\frac{\int_0^t\sigma_u^2e^{2\int_0^u\theta_vdv}du}{e^{2\int_0^t\theta_udu}}\\
&=-\frac{1}{2}z^2\cdot\lim\limits_{t\to\infty}\frac{\sigma_t^2}{2\theta_t}\\
&=-\frac{1}{2}z^2\frac{b(x)}{\al(x,x)}.
\eea
\eeq
Finally, we conclude that $(\eta_t)_{t\geq0}$ has stationary distribution $\mathcal{N}(0,\frac{b(x)}{\al(x,x)})$.
\end{proof}

%
%
%
%

\end{document}